\documentclass[10pt,a4paper,oneside,english]{amsart}
\usepackage[T1]{fontenc}
\usepackage[latin9]{inputenc}
\usepackage{units}
\usepackage{amsthm}
\usepackage{amssymb}

\makeatletter

\numberwithin{equation}{section}
\numberwithin{figure}{section}
\theoremstyle{plain}
\newtheorem{thm}{\protect\theoremname}[section]
  \theoremstyle{plain}
  
  \theoremstyle{plain}
  
  \theoremstyle{remark}
  \newtheorem{rem}[thm]{\protect\remarkname}
  \theoremstyle{plain}
  \newtheorem{cor}[thm]{\protect\corollaryname}

\@ifundefined{date}{}{\date{}}

\usepackage{amsthm}

\usepackage{cite}

\newtheorem{theorem}{Theorem}[section]

\newtheorem{example}[theorem]{Example}

\providecommand{\lemmaname}{Lemma}
  \providecommand{\propositionname}{Proposition}
  \providecommand{\remarkname}{Remark}
\providecommand{\theoremname}{Theorem}

\usepackage{babel}
\providecommand{\lemmaname}{Lemma}
  \providecommand{\propositionname}{Proposition}
  \providecommand{\remarkname}{Remark}
\providecommand{\theoremname}{Theorem}

\usepackage{babel}
\providecommand{\corollaryname}{Corollary}
  \providecommand{\lemmaname}{Lemma}
  \providecommand{\propositionname}{Proposition}
  \providecommand{\remarkname}{Remark}
\providecommand{\theoremname}{Theorem}

\makeatother

\usepackage{babel}
  \providecommand{\corollaryname}{Corollary}
  \providecommand{\lemmaname}{Lemma}
  \providecommand{\propositionname}{Proposition}
  \providecommand{\remarkname}{Remark}
\providecommand{\theoremname}{Theorem}

\begin{document}

\title[A remark on spaces of affine continuous functions on a simplex] {A remark on spaces of affine continuous functions on a simplex}

\author{E. Casini}

\address{Dipartimento di Scienza e Alta Tecnologia, Università dell'Insubria,
via Valleggio 11, 22100 Como, Italy }

\email{emanuele.casini@uninsunbria.it}

\author{E. Miglierina}

\address{Dipartimento di Discipline Matematiche, Finanza Matematica ed Econometria,
Università Cattolica del Sacro Cuore, Via Necchi 9, 20123 Milano,
Italy }

\email{enrico.miglierina@unicatt.it}

\author{\L. Piasecki}

\address{Instytut Matematyki,
Uniwersytet Marii Curie-Sk{\l}odowskiej, Pl. Marii Curie-Sk{\l}odowskiej 1, 20-031 Lublin, Poland}

\email{piasecki@hektor.umcs.lublin.pl}
\begin{abstract}
We present an example of an infinite dimensional separable space of
affine continuous functions on a Choquet simplex that does not
contain a subspace linearly isometric to $c$. This example disproves a result stated in \cite{Zippin1969}.
\end{abstract}

\subjclass[2010]{Primary 46B04; Secondary 46B45, 46B25}

\keywords{Affine functions, Lindenstrauss spaces, Space of convergent sequences, Polyhedral spaces}

\maketitle

\section{Introduction and Preliminaries}
In the {\it Concluding remarks} in \cite{Zippin1969}, the author claims that a
separable predual of an abstract $L_1$
space contains a (complemented) copy of $c$
(the Banach space of real convergent sequences) if its unit ball has an extreme point.
Only a sketch of the proof of this property was indicated.
In particular there is no proof that the extreme point and the sequence $\{y_n\}$ ($1$-equivalent to the standard
basis of $c_0$ built in the main theorem of \cite{Zippin1969}), span a subspace isometric to $c$.
The aim of this paper is to present a simple example that disproves Zippin's claim.
Moreover, in the last section of our paper we point out that our example also shows that some known results, establishing geometrical properties of polyhedral Banach spaces, are incorrect.

Let $B_X$ ($S_X$) denote the closed unit ball (sphere) in a real Banach space $X$ and
$X^*$ denotes the dual of $X$. If $K$ is a compact, convex subset of
a linear topological space, then by $\textrm{Ext }K$ we denote the
set of all extreme points of $K$. A convex subset $F$ of $B_X$ is
called a \textit{face} of $B_X$ if for every $x,y\in B_X$ and
$\lambda \in (0,1)$ such that $(1-\lambda)x+ \lambda y \in F$ we
have $x,y\in F$. A face $F$ of $B_X$ is named a \textit{proper face}
if $F\neq B_X$. Here $c$ denotes the Banach space of all real
convergent sequences and $A(K)$ stands for a \textit{simplex space},
that is, the space of all affine continuous functions on a Choquet
simplex $K$ endowed with the supremum norm. It is well known that
$c^*=\ell_1$ and the duality is given by:
$$
f(x)=f(1)\lim x(i)+\sum_{i=1}^{+\infty}f(i+1)x(i)
$$
where $f=(f(1),f(2),\dots)\in \ell_1$ and $x=(x(1),x(2),\dots)\in c$.
A Banach
space $X$ is called an $L_1$-\textit{predual space} or a
\textit{Lindenstrauss space} if its dual is isometric to $L_1(\mu)$
for some measure $\mu$. It is well known that this class includes
all the simplex spaces. Moreover a Lindenstrauss space $X$ is isometric to a simplex space
if and only if $B_X$ has at least one extreme point (see \cite{Semadeni1964}). Finally, we recall that a Banach space $X$ is polyhedral if the unit balls of all its finite-dimensional subspaces are polytopes (see \cite{Klee}).

\section{A simplex space not containing $c$}

We begin by providing a necessary condition for the presence
of a copy of $c$ in a separable Banach space $X$.

\begin{thm}\label{c in a separable Banach space}
Let $X$ be a separable Banach space. If $X$ contains a subspace
linearly isometric to $c$, then there exist $x\in X$ and a sequence
$(e_n^*)\subset \textrm{Ext }B_{X^*}$ such that $(e_n^*)$ is
$w^*$-convergent to $e^*$,
$e_n^*(x)=e^*(x)=\left\|e^*\right\|=\left\|x\right\|=1$ and
$\left\|e_n^*\pm e^*\right\|=2$ for every $n\in \mathbb{N}$.
\end{thm}

\begin{proof}
Assume that $X$ contains an isometric copy of $c$. Consider
$$x=(1,1,\dots,1,\dots)\in c,$$
the sequence $(x_n)_{n\in\mathbb{N}}\subset B_c$ defined by
\begin{eqnarray*}
&&x_1=(-1,1,1,\dots,1,\dots),\\
&&x_2=(1,-1,1,1,\dots,1,\dots),\\
&&x_3=(1,1,-1,1,1,\dots,1,\dots),\\
&&\dots
\end{eqnarray*}
and the sequence $(x_n^*)_{n\in \mathbb{N}}\subset B_{c^*}$ defined
by
\begin{eqnarray*}
&&x_1^*=(0,1,0,0,\dots,0,\dots),\\
&&x_2^*=(0,0,1,0,0,\dots,0,\dots),\\
&&x_3^*=(0,0,0,1,0,0,\dots,0,\dots),\\
&& \dots \quad .
\end{eqnarray*}
Let $\widetilde{x_n^*}$ denote a norm preserving linear extension of
$x_n^*$ to the whole $X$.
Next, let us define the sets $F_n$, $n\in \mathbb{N}$, by
$$F_n=\left\{x^*\in B_{X^*}: x^*(x_n)=-1 \textrm{ and } x^*(x_m)=x^*(x)=1 \textrm{ for } m\neq n\right\}.$$
It is easy to see that
\begin{itemize}
  \item[(a)] $F_n\neq \emptyset$ for every $n\in\mathbb{N}$ (because $\widetilde{x_n^*}\in
  F_n$),
  \item[(b)] $F_n$ is a $w^*$-closed proper face of $B_{X^*}$, for
  every $n\in\mathbb{N}$,
  \item[(c)] $F_n \cap F_m = \emptyset$ provided $m\neq n$.
\end{itemize}
Hence, $ \textrm{Ext }F_n\neq \emptyset$ by the Krein-Milman
Theorem, $ \textrm{Ext }F_n\subset \textrm{Ext }B_{X^*}$ for every
$n\in\mathbb{N}$ and $\textrm{Ext }F_n \cap \textrm{Ext }F_m=
\emptyset$ whenever $m \neq n$.

Let $e_n^*\in \textrm{Ext }F_n\subset \textrm{Ext }B_{X^*}$. We can
assume that $(e_n^*)$ is $w^*$-convergent, let us say to $e^*$. Then
\begin{itemize}
  \item [(d)]$e_n^*(x)=e^*(x)=\left\|e^*\right\|=\left\|x\right\|=1$ for
  every $n\in\mathbb{N}$,
  \item [(e)] $e^*(x_i)=\lim\limits_{n}e_n^*(x_i)=1$ for every $i\in\mathbb{N}$.
\end{itemize}
Consequently, for every $n\in \mathbb{N}$, we have
$$2\geq \left\|e^*-e_n^*\right\|\geq e^*(x_n)-e_n^*(x_n)=1-(-1)=2$$
and
$$2\geq \left\|e^*+e_n^*\right\|\geq e^*(x)+e_n^*(x)=1+1=2.$$

\end{proof}
The previous theorem gives a corollary which allows us to show that the announced example does not contain $c$.

\begin{cor}\label{c in l_1 predual}
Let $X$ be a predual of $\ell_1$. If $X$ contains a subspace
isometric to $c$ then there exist $x\in B_X$ and a subsequence
$(e_{n_k}^*)_{k\in \mathbb{N}}$ of the standard basis
$(e_{n}^*)_{n\in \mathbb{N}}$ in $\ell_1$ such that
\begin{itemize}
\item[(1)] $e_{n_k}^* \overset{\sigma(\ell_1,X)}{\longrightarrow} e^*$ and $\textrm{supp }e_{n_k}^* \cap \textrm{supp } e^*= \emptyset$ for
  every $k\in\mathbb{N}$, where for $x^* \in \ell_1=X^*$ we put $\textrm{supp }x^*:=\left\{i\in \mathbb{N}:x^*(i)\neq
  0\right\}$,
  \item[(2)] $e_{n_k}^*(x)=e^*(x)=1$ for every
  $k\in\mathbb{N}$.
\end{itemize}
\end{cor}

\begin{example}\label{example W_f}
Let
$$
W=\left\{x=(x(1),x(2),\dots)\in c:
\lim_{i}x(i)=\sum_{i=1}^{\infty}\frac{x(i)}{2^{i}}\right\}.
$$

The hyperplane $W$ has the following properties:
\begin{itemize}
  \item [(a)]   The map $\phi:\ell_{1}\rightarrow
  W^{*}$ defined by
\[
(\phi(y))(x)=\sum_{j=1}^{+\infty}x(j)y(j),
\]
where $y=(y(1),y(2),\dots)\in\ell_{1}$ and
$x=(x(1),x(2),\dots)\in W$ is an onto isometry. Moreover, if $\left(
e_{n}^*\right) $ denotes the standard basis of $\ell_{1}$, then
\[
e_{n}^*\overset{\sigma(\ell_{1},W)}{\longrightarrow}e^*=\left(\frac12,\frac{1}{2^2},\frac{1}{2^3},\dots\right)
\]
(see Theorem 4.3 in \cite{Casini-Miglierina-Piasecki2014}).

\item [(b)] From Corollary \ref{c in l_1 predual} we conclude that $W$ does not contain a subspace linearly isometric
to $c$.

\item [(c)] By Corollary 2 in \cite{Japon-Prus2004} the set
$$K=\left\{(y(1),y(2),\dots)\in \ell_1: \sum_{i=1}^{\infty}y(i)=1, y(i)\geq 0, i=1,2,\dots\right\}$$
is an infinite dimensional $\sigma(\ell_{1},W)$-closed proper
face of $B_{\ell_1}$.

\item [(d)] It is easy to see that $x=(1,1,\dots,1,\dots)\in \textrm{Ext
}B_{W}$. Consequently, as was observed in \cite{Semadeni1964},
$W$ is isometric to $A(K)$. Nevertheless, in our
special case this property can be shown directly.

\item[(e)] In order to prove that the space $W$ is polyhedral we need a characterization of polyhedrality given by Durier and Papini (Theorem 2 in \cite{Durier-Papini}): a Banach space $X$ is polyhedral if and only if the set
\[
C(x)=\{y\in X:\exists \lambda >0, \|x+\lambda(y-x)\|\leq 1\}
\]
is a closed set for every $x\in S_X$. Moreover, we remark that $x\in S_W$ if and only if there exists at least one index $i_0 \in \mathbb{N}$ such that $\left|x(i_0)\right|=1$. Then, an easy computation shows that 
\[
C(x)=\left\lbrace  y\in W: y(i)\leq 1 \text{ for } i\in I(x) \text{ and }  y(j)\geq -1 \text{ for } j\in J(x) \right\rbrace,
\]
where $I(x)=\{i:x(i)=1\}$ and $J(x)=\{j:x(j)=-1\}$. Therefore the set $C(x)$ is closed for every $x \in S_W$.

\end{itemize}
\end{example}

\begin{rem}
Example \ref{example W_f} shows that property (2) in Corollary \ref{c in l_1 predual} does not imply
that $c \subset X$. Also property (1) in the same corollary does not imply
that $c \subset X$. Indeed, to this end is sufficient to consider a different hyperplane of $c$:
$$V=\left\{x=(x(1),x(2),\dots)\in c:
\lim_{i}x(i)=\sum_{i=1}^{\infty}\frac{(-1)^{i+1}x(2i-1)}{2^{i}}\right\}.$$
By using Theorem 4.3 in \cite{Casini-Miglierina-Piasecki2014} we
have that $V^*=\ell_1$ and
$$e_{2n}^* \overset{\sigma(\ell_1,V)}{\longrightarrow} e^*=\left(\frac12,0,-\frac14,0,\frac18,0,-\frac{1}{16},\dots\right).$$
It is easy to see that does not exist $x\in V$ satisfying
the property (2) in Corollary \ref{c in l_1 predual}. Therefore
$V$ does not contain an isometric copy of $c$.
It would be desirable to understand if the simultaneous validity of conditions (1) and (2)  ensures the presence of a isometric copy of $c$ in a predual of $\ell_1$, but we have not been able to do this.
Nevertheless we show that the necessary condition expressed in Theorem \ref{c in a separable Banach space} is not
a sufficient condition in a general framework. Indeed, let us consider the space $X=\ell_1$ and the sequence $(x_n^*)_{n\in
\mathbb{N}}$ in $\ell_1^*=\ell_{\infty}$ defined by
\begin{eqnarray*}
&&x_1^*=(1,-1,1,1,\dots,1,\dots),\\
&&x_2^*=(1,1,-1,1,1,\dots,1,\dots),\\
&&x_3^*=(1,1,1,-1,1,1,\dots,1,\dots),\quad \dots.
\end{eqnarray*}
Then

\begin{itemize}
  \item[(a)] $x_n^* \overset{\sigma(\ell_{\infty},\ell_1)}{\longrightarrow} x^*=(1,1,\dots,1,\dots)$,
  \item[(b)] $\left\|x_n^* \pm x^*\right\|=2$ for
  every $n\in\mathbb{N}$,
  \item[(c)] $x_n^*, x^*\in \textrm{Ext }B_{\ell_{\infty}}$ for
  every $n\in\mathbb{N}$,
  \item[(d)] for $e_1=(1,0,0,\dots,0,\dots)\in \ell_1$ we have $x^*(e_1)=x_n^*(e_1)=1$
\end{itemize}
but $\ell_1$ does not contain $c$.
\end{rem}

\section{Final remarks}

Different authors refer to Zippin's statement. We focus on a paper by Lazar that gives a characterization of polyhedral Lindenstrauss spaces.

In \cite{Lazar1969}, the implication (1)$\Rightarrow$(3) in Theorem
3 is incorrect. Indeed, $W$ is a polyhedral space but $B_{W^*}$ contains an infinite dimensional $w^*$-closed
proper face (see items (e) and (c) in Example \ref{example W_f}). 
As a consequence of this remark we have that some of the implications stated in the theorem mentioned above reveal to be unproven. For instance, the implication (1)$\Rightarrow$(4) has no proof.

The result of Lazar has been subsequently used by Gleit and McGuigan in \cite{Gleit-McGuigan1972} to provide another characterization of polyhedral Lindenstrauss spaces. We remark that, in \cite{Gleit-McGuigan1972}, the implication
(3)$\Rightarrow$(1) in Theorem 1.2 is incorrect. Indeed, $W$ does
not contain an isometric copy of $c$ and for $x=(1,1,\dots,1,\dots)$
in $W$ and the $w^*$-limit
$e^*=\left(\frac12,\frac{1}{2^2},\frac{1}{2^3},\dots\right)$ of the
standard basis in $\ell_1=W^*$ we have (see items (b) and (a) in Example
\ref{example W_f})
$$e^*(x)=\sum_{i=1}^{\infty}\frac{1}{2^i}=1=\left\|x\right\|=\left\|e^*\right\|.$$
Also, the implication (3)$\Rightarrow$(1) in Corollary 2.7 is
incorrect because $W$ is a simplex space (see item (d) in Example
\ref{example W_f}).


\end{document}